\documentclass{amsart}
\usepackage{mathrsfs}
\usepackage{bm}
\usepackage{enumerate}

\newtheorem{theorem}{Theorem}[section]

\newtheorem{proposition}[theorem]{Proposition}

\theoremstyle{definition}
\newtheorem{definition}[theorem]{Definition}
\newtheorem{remark}[theorem]{Remark}
\newtheorem{example}[theorem]{Example}

\theoremstyle{remark}

\numberwithin{equation}{section}

\newcommand{\e}{\epsilon}

\newcommand{\mis}{\mathfrak{M}}

\newcommand{\tp}{\otimes}
\newcommand{\htp}{\mathbin{\hat\otimes}}

\newcommand{\C}{\mathbb{C}}

\newcommand{\cL}{\mathcal{L}}

\newcommand{\A}{{\bm A}}

\newcommand{\fA}{\mathfrak{A}}

\renewcommand{\c}{\mathscr{C}}

\DeclareMathAlphabet{\mathbbmsl}{U}{bbm}{m}{sl}

\newcommand{\p}{\mathcal{P}}
\newcommand{\pp}{\mathbbmsl{P}}

\newcommand{\ip}[1]{\langle #1 \rangle}

\newcommand{\N}{\mathbb{N}}

\newcommand{\U}{\mathbf{1}}

\newcommand{\set}[1]{\{#1\}}

\newcommand{\Bigset}[1]{\Bigl\{ #1 \Bigr\}}

\newcommand{\Bigabs}[1]{\Bigl\lvert #1 \Bigr\rvert}

\newcommand{\bigprn}[1]{\bigl( #1 \bigr)}
\newcommand{\Bigprn}[1]{\Bigl( #1 \Bigr)}
\newcommand{\biggprn}[1]{\biggl( #1 \biggr)}

\newcommand{\enorm}{\lVert\,\cdot\,\rVert}
\newcommand{\norm}[1]{\lVert #1 \rVert}

\newcommand{\Bignorm}[1]{\Bigl\lVert #1 \Bigr\rVert}

\newcommand{\lVertt}{\lvert\mspace{-2mu}\lvert\mspace{-2mu}\lvert}
\newcommand{\rVertt}{\rvert\mspace{-2mu}\rvert\mspace{-2mu}\rvert}
\newcommand{\enormm}{\lVertt\, \cdot \,\rVertt}
\newcommand{\normm}[1]{\lVertt #1 \rVertt}

\newcommand{\BiglVertt}{\Bigl\lvert\mspace{-3mu}\Bigl\lvert\mspace{-3mu}\Bigl\lvert}
\newcommand{\BigrVertt}{\Bigr\rvert\mspace{-3mu}\Bigr\rvert\mspace{-3mu}\Bigr\rvert}

\newcommand{\Bignormm}[1]{\BiglVertt #1 \BigrVertt}

\begin{document}

\title{Algebras of Polynomials Generated by Linear Operators}

\author{F. Zaj and M. Abtahi${}^*$}

\address{School of Mathematics and Computer Sciences,
Damghan University, Damghan, P.O.BOX 36715-364, Iran}

\thanks{${}^*$ Corresponding author (abtahi@du.ac.ir)}
\email{abtahi@du.ac.ir (M. Abtahi)}



\maketitle

\begin{abstract}
  Let $E$ be a Banach space and $A$ be a commutative Banach algebra with identity. Let
  $\pp(E, A)$ be the space of $A$-valued polynomials on $E$ generated by
  bounded linear operators (an $n$-homogenous polynomial in $\pp(E,A)$ is
  of the form $P=\sum_{i=1}^\infty T^n_i$, where $T_i:E\to A$
  ($1\leq i <\infty$) are bounded linear operators and
  $\sum_{i=1}^\infty \|T_i\|^n < \infty$). For a compact set $K$ in $E$,
  we let $\pp(K, A)$ be the closure in $\c(K,A)$ of the restrictions $P|_K$
  of polynomials $P$ in $\pp(E,A)$. It is proved that $\pp(K, A)$ is
  an $A$-valued uniform algebra and that, under certain conditions,
  it is isometrically isomorphic to the injective tensor product $\p_N(K)\htp_\e A$,
  where $\p_N(K)$ is the uniform algebra on $K$ generated by nuclear scalar-valued
  polynomials. The character space of $\pp(K, A)$ is then identified with $\hat{K}_N\times \mis(A)$,
  where $\hat K_N$ is the nuclear polynomially convex hull of $K$ in $E$, and
  $\mis(A)$ is the character space of $A$.
\end{abstract}

\section{Introduction}

Let $E$ be a Banach space and $A$ be a commutative unital Banach algebra, over the complex field $\C$.
For a compact set $K$ in $E$, let $\c(K,A)$ be the algebra of all continuous functions
$f:K\to A$ equipped with uniform norm
\begin{equation}\label{eqn:uniform-norm}
  \|f\|_K = \sup\set{\|f(x)\|: x\in K}.
\end{equation}

When $A=\C$, we write $\c(K)$ instead of $\c(K,\C)$.

Given an element $a\in A$, the same notation $a$ is used for the constant
function given by $a(x)=a$, for all $x\in K$, and $A$ is regarded
as a closed subalgebra of $\c(K,A)$. We denote by $\U$ the unit element of $A$,
and identify $\C$ with the closed subalgebra $\C\U=\set{\alpha\U:\alpha\in\C}$ of $A$.
Therefore, every function $f\in\c(K)$ can be seen as the $A$-valued function $x\mapsto f(x)\U$.
We use the same notation $f$ for this $A$-valued function, and regard $\c(K)$
as a closed subalgebra of $\c(K,A)$. By an \emph{$A$-valued uniform algebra on $K$}
we mean a closed subalgebra $\A$ of $\c(K,A)$ that contains the constant functions
and separates points of $K$; see \cite{Abtahi-BJMA,Nikou-Ofarrell}. A comprehensive discussion
on complex function algebras appears in \cite[Chapter 4]{Dales}.

We are mostly interested in those $A$-valued
uniform algebras that are invariant under composition with characters of $A$:
an $A$-valued uniform algebra $\A$ is called \emph{admissible} if
$\phi\circ f \in \A$ whenever $f\in \A$ and $\phi:A\to\C$ is a character of $A$.
Recall that a character is just a nonzero multiplicative linear functional.
Denoted by $\mis(A)$, the set of all characters of $A$, equipped with the Gelfand topology,
is a compact Hausdorff space.

When $\A$ is admissible, we let $\fA$ be the subalgebra of $\A$ consisting of
scalar-valued functions; that is, $\fA=\A\cap\c(K)$. In this case,
$\fA=\set{\phi\circ f : f\in\A}$, for every $\phi\in\mis(A)$.
The algebra $\c(K,A)$ is admissible with $\fA=\c(K)$. It is well-known that
$\c(K,A)$ is isometrically isomorphic to the injective tensor product $\c(K)\htp_\e A$,
and that $\mis(\c(K,A))=K\times \mis(A)$; see \cite{Hausner}. In general, given
an admissible $A$-valued uniform algebra $\A$, it is natural to ask whether
the analogous equalities $\A=\fA\htp_\e A$ and $\mis(\A)=\mis(\fA)\times\mis(A)$ hold.
In this paper, we investigate these questions for a certain $A$-valued uniform algebra
that is generated by linear operators.

\begin{remark}\label{rem:injective-tesor-prod}
  For an admissible $A$-valued uniform algebra $\A$ with $\fA=\c(K)\cap \A$,
  let $\fA A$ denote the subalgebra of $\A$ generated by $\fA\cup A$.
  Indeed, $\fA A$ consists of elements of the form $f=f_1 a_1+\dotsb+f_na_n$,
  where $f_i\in\fA$ and $a_i\in A$, for $1\leq i \leq n$, $n\in\N$. The following statements
  are equivalent.
  \begin{enumerate}
    \item $\A=\fA\htp_\e A$ (isometrically isomorphic).
    \item The subalgebra $\fA A$ is dense in $\A$.
  \end{enumerate}
  In fact, the mapping $\Lambda_0:f\tp a \mapsto fa$ defines a homomorphism
  of $\fA\tp A$ onto $\fA A$, and, using Hahn-Banach theorem, we have
    \begin{align*}
    \Bignorm{\Lambda_0\Bigl(\sum_{i=1}^n f_i\tp a_i\Bigr)}_K
    & = \sup_{x\in K} \Bignorm{\sum_{i=1}^n f_i(x)a_i}
      = \sup_{x\in K} \sup_{\phi\in A^*_1}
        \Bigabs{\sum_{i=1}^n f_i(x)\phi(a_i)} \\
    & = \sup_{\phi\in A^*_1} \Bignorm{\sum_{i=1}^n f_i(\cdot) \phi(a_i)}_K
      = \sup_{\phi\in A^*_1}\sup_{\psi\in\fA^*_1}
        \Bigabs{\psi\Bigl(\sum_{i=1}^n f_i(\cdot)\phi(a_i)\Bigr)} \\
    & = \sup_{\psi\in\fA^*_1}\sup_{\phi\in A^*_1}
        \Bigabs{\sum_{i=1}^n \psi(f_i) \phi(a_i)}
      = \Bignorm{\sum_{i=1}^n f_i\tp a_i}_\epsilon,
  \end{align*}
  where $\fA^*_1$ and $A^*_1$ denote the unit balls of $\fA^*$ and $A^*$, respectively.
  Therefore, $\Lambda_0$ extends to an isometry of $\fA\htp_\e A$ onto $\overline{\fA A}$.
  Given $f\in\fA$ and $a\in A$, we may identify the $A$-valued function $fa$ with
  the elementary tensor $f\tp a$.
\end{remark}

\begin{remark}[\cite{Lourenco-a-class}] \label{rem:E-is-an-algebra}
  It is always possible to define a product $(a,b)\mapsto ab$ on the Banach space
  $E$ in order to make it an algebra with identity. Indeed, take a nonzero functional
  $\psi\in E^*$ and a vector $e \in E$ such that $\psi(e) \neq 0$ and $\|e\| = 1$.
  Then $E=\ker\psi\oplus \C e$. For $a,b\in E$, write $a = x + \psi(a)e$ and $b = y + \psi(b)e$,
  with $x, y \in \ker\psi$, and define $ab = \psi(b)x + \psi(a)y + \psi(a)\psi(b)e$. This makes $E$
  an algebra with identity $e$. Define a norm on $E$ by $\norm{a}_1 = |\psi(a)| + \|x\|$,
  where $a = x + \psi(a)e$. Then, $\enorm_1$ is an equivalent norm on $E$ making it a commutative
  unital Banach algebra. This observation allows us to consider $\c(K,E)$ as a Banach algebra,
  even if the Banach space $E$ is not assumed to be an algebra in the first place.
\end{remark}

In this paper, we consider the space $\pp(E, A)$ of all $A$-valued polynomials
on $E$ generated by bounded linear operators $T:E\to A$. By definition, an $n$-homogenous
polynomial $P$ in $\pp(E,A)$ is of the form $P=\sum_{i=1}^\infty T_i^n$, where
$(T_i)$ is a sequence of bounded linear operators of $E$ into $A$ such that
$\sum_{i=1}^\infty \|T_i\|^n< \infty$. This class of polynomials was introduced
in \cite{Lourenco-a-class} with further study carried out in \cite{Galindo-PGLO}.
For a compact set $K$ in $E$, we let $\pp(K, A)$ be the closure in $\c(K,A)$ of
the restrictions $P|_K$ of polynomials $P\in\pp(E,A)$. It is proved that $\pp(K, A)$ is
an admissible $A$-valued uniform algebra on $K$ with $\fA=\p_N(K)$,
where $\p_N(K)$ represents the (complex) uniform algebra on $K$ generated by nuclear polynomials.
In fact, $f\in\p_N(K)$ if, and only if, there is a sequence $(P_k)$ of nuclear
scalar-valued polynomials on $E$ such that $P_k\to f$ uniformly on $K$.
We prove that the following are equivalent.
\begin{enumerate}[(i)]
 \item $\pp(K,A)=\p_N(K)\htp_\e A$, for all Banach algebra $A$,
 \item $\pp(K,E)=\p_N(K)\htp_\e E$ (see Remark \ref{rem:E-is-an-algebra}),
 \item $I\in \p_N(K)\htp_\e E$, where
$I:E\to E$ is the identity operator.
\end{enumerate}
In this situation, the character space
$\mis(\pp(K, A))$ is identified with $\hat{K}_N\times \mis(A)$,
where $\hat K_N$ denotes the nuclear polynomially convex hull of $K$ in $E$.

\section{Algebras of Polynomials generated by linear operators}

First, let us recall basic definitions, notations and some results of the theory
of polynomials on Banach spaces. For comprehensive texts, see \cite{Dineen}, \cite{Mujica}
and \cite{Prolla}.

Let $E$ and $F$ be  Banach spaces over $\C$. The space of all continuous symmetric
$n$-linear operators $T:E^n\to F$ is denoted by $\cL_s({}^nE,F)$. For $T\in \cL_s({}^nE,F)$,
define $\hat T(x) = T(x,\dotsc,x)$, $x\in E$. A mapping $P :E\to F$ is said to be an $n$-homogeneous
polynomial if $P=\hat T$, for some $T\in \cL_s({}^nE,F)$. For convenience, $0$-homogeneous
polynomials are defined as constant mappings from $E$ into $F$. The vector space of all
$n$-homogeneous polynomials from $E$ into $F$ is denoted by $\p({}^nE,F)$.
The shortened notation $\p({}^nE)$ is used when $F=\C$. A norm on $\p({}^nE,F)$ is defined as
\begin{equation}\label{eqn:norm-of-P}
  \|P\| = \sup\set{\|P(x)\|:\|x\|\leq 1},\ P\in \p({}^nE,F).
\end{equation}

Continuity of $P$ (and of the corresponding symmetric $n$-linear operator $T$) is then equivalent
to finiteness of $\|P\|$. Given an $n$-homogeneous polynomial
$P\in\p({}^nE,F)$, the symmetric $n$-linear mapping
$T$ that gives rise to $P$ can be recovered by any of several polarization formulae;
\begin{equation}\label{eqn:polarization}
  T(x_1,\dotsc,x_n) =
     \frac1{2^n n!} \sum_{\e_j=\pm1}\e_1\dotsm\e_n P\Bigprn{\sum_{j=1}^n \e_j x_j}.
\end{equation}
As a special case, for $a,b\in A$ and $m,n\in \N$, we get
\begin{equation}\label{eqn:polarization-ab}
    a^m b^n = \frac1{2^{m+n}(m+n)!}\sum_{\e_\ell=\pm 1}
     \e_1\dotsb\e_{m+n}\biggprn{\sum_{\ell=1}^m\e_\ell a + \sum_{\ell=m+1}^{m+n}\e_\ell b}^{m+n}.
\end{equation}

Therefore, $n$-homogeneous polynomials and symmetric $n$-linear operators
are in one-to-one correspondence, and the polarization inequality
$\|P\| \leq \|T\| \leq \frac{n^n}{n!}\|P\|$ shows that these spaces are isomorphic;
see \cite[Corollary 1.7 and Proposition 1.8]{Dineen}.

\paragraph{Notation.}
  Let $T\in \cL_s({}^nE,F)$ and $x,y\in E$. For $0 \leq k \leq n$, let
  \begin{equation}\label{eqn:notation-T(x,..,x)}
    T (x^k,y^{n-k}) = T(\underbrace{x,x,\dotsc,x}_{k\text{ times}},\underbrace{y,y,\dotsc,y}_{n-k\text{ times}}).
  \end{equation}

\noindent
Then, by \cite[Theorem 1.8]{Mujica}, we have the \emph{Leibniz formula}
\begin{equation}\label{eqn:Leibniz-formula}
  T\bigprn{(x+y)^n} =
    \sum_{k=1}^n \binom{n}{k} T(x^k,y^{n-k}).
\end{equation}

\subsection{The $A$-valued uniform algebra generated by linear operators}

To continue, we restrict ourselves to Banach algebra valued polynomials.
We let $A$ be a commutative Banach algebra with identity $\U$.
Adopting notations from \cite{Lourenco-a-class}, we denote by $\pp({}^n E,A)$  the space
of $n$-homogeneous $A$-valued polynomials on $E$ of the form
\begin{equation}\label{eqn:dfn-of-P_n}
  P(x)=\sum_{i=1}^\infty T_i(x)^n,\ x\in E,
\end{equation}
where $(T_i)$ is a sequence in $\cL(E, A)$ such that $\sum_{i=1}^\infty \|T_i\|^n <\infty$.
A norm on $\pp({}^nE,A)$ is defined by
\begin{equation}\label{eqn:normm-of-P}
 \normm{P} = \inf \Bigset{\sum_{i=1}^\infty \|T_i\|^n: P=\sum_{i=1}^\infty T_i^n},
\end{equation}
where the infimum is taken over all possible representations of $P$ in \eqref{eqn:dfn-of-P_n}.
It is easy to verify that $\enormm$ is a norm and that $\|P\|\leq \normm{P}$, for all
$P\in\pp({}^nE,A)$. The following shows that convergence with respect to this norm
implies uniform convergence on compact sets.

\begin{proposition}\label{prop:norm-u-leq-M-normm}
  Let $K$ be a compact set in $E$. Then, there is $M>0$ such that
    \begin{equation}\label{eqn:norm-K-and-normm}
       \|P\|_K \leq M^n\normm{P},\ P\in\pp({}^nE,A).
    \end{equation}
    Consequently, the series in \eqref{eqn:dfn-of-P_n} converges uniformly on $K$.
\end{proposition}

\begin{proof}
  Since $K$ is compact, there is a constant $M$ such that $\|x\|\leq M$,
  for all $x\in K$. Therefore, $\|Tx\| \leq M\|T\|$, for every $x\in K$ and
  $T\in\cL(E,A)$. If a polynomial $P$ is defined by \eqref{eqn:dfn-of-P_n}, then
  \begin{equation*}
    \|P\|_K = \sup_{x\in K}\Bignorm{\sum_{i=1}^\infty T_i(x)^n}
    \leq \sup_{x\in K}\sum_{i=1}^\infty \|T_i(x)\|^n \leq M^n \sum_{i=1}^\infty \|T_i\|^n.
  \end{equation*}

  The above inequality holds for any representation of $P$ in \eqref{eqn:dfn-of-P_n}.
  Taking infimum over all those representations, as in \eqref{eqn:normm-of-P},
  we get $\|P\|_K \leq M^n \normm{P}$. Also, we have
  \begin{multline*}
    \Bignorm{P - \sum_{i=1}^s T_i^n}_K = \Bignorm{\sum_{i=s+1}^\infty T_i^n}_K
       \leq M^n\Bignormm{\sum_{i=s+1}^\infty T_i^n}
       \leq M^n \sum_{i=s+1}^\infty \|T_i\|^n \to0\quad (s\to\infty).
  \end{multline*}

  This shows that the series in \eqref{eqn:dfn-of-P_n} converges uniformly on $K$.
\end{proof}

If $A$ is replaced by $\C$, we reach the nuclear (scalar-valued) polynomials.
By definition, an $n$-homogenous polynomial $P:E\to\C$ is called \emph{nuclear} if it
can be written in a form
\begin{equation}\label{eqn:nuclear-polynomials}
  P(x) = \sum_{i=1}^\infty \psi_i(x)^n,\ x\in E,
\end{equation}
where $(\psi_i)$ is a sequence in $E^*$ with $\sum_{i=1}^\infty \|\psi_i\|^n < \infty$.
We denote by $\p_N({}^n E)$, the space of all $n$-homogenous nuclear (scalar-valued) polynomials
on $E$. In fact, $\p_N({}^nE)=\pp({}^nE,\C)$.  Nuclear polynomials between Banach spaces have been
studied by many authors; see \cite{Carando-extendible}, \cite{Carando-duality},
\cite{Cilia-nuclear} and \cite{Zal-survay}.

We are now in a position to introduce the $A$-valued uniform algebras that these polynomials
generate. We let $\pp(E,A)$ be the space of all polynomials on $E$ of the form
$P=P_0+P_1+\dotsb+P_n$, where $P_k\in \pp({}^kE,A)$, $0\leq k \leq n$, $n\in\N$.
In the same fashion, the space $\p_N(E)$ of all linear combinations of homogenous nuclear
polynomials on $E$ is defined. In fact, $\p_N(E) = \pp(E,\C)$.

\begin{definition}\label{dfn:P(K,A)}
  Let $K$ be a compact set in the Banach space $E$. Define
  \begin{align*}
    \pp_0(K,A)  = \set{P|_K : P \in \pp(E,A)}, \quad
    \p_{N_0}(K) = \set{P|_K : P \in \p_N(E)}.
  \end{align*}
  We define $\pp(K,A)$ as the closure of $\pp_0(K,A)$ in $\c(K, A)$.
  Similarly, $\p_N(K)$ is defined to be the closure of $\p_{N_0}(K)$ in $\c(K)$.
\end{definition}

Therefore, $f\in\pp(K,A)$ (resp.\ $f\in \p_N(K)$) if, and only if, there is a sequence $(P_k)$ of polynomials
in $\pp(E,A)$ (resp.\ $\p_{N}(E)$) such that $P_k\to f$ uniformly on $K$.

Clearly, $\pp(K,A)$ is a closed subspace of $\c(K,A)$. We are aiming to show that
$\pp(K, A)$ is a subalgebra of $\c(K,A)$. To achieve this, one may think of proving
that the space $\pp(E,A)$ itself is an algebra (that is, if $P\in\pp({}^mE,A)$ and
$Q\in\pp({}^nE,A)$ then $PQ\in\pp({}^{m+n}E,A)$).
This manner is naturally expected. However, the authors do not currently have
strong evidence supporting or opposing the possibility that $\pp(E,A)$ is an algebra.
Our approach, therefore, is to give a direct proof of the fact that $\pp(K, A)$ is
an algebra, as follows.

\begin{theorem}\label{thm:P(K,A)-is-an-algebra}
  For a compact set $K$ in $E$, let $\A=\pp(K,A)$.
  Then $\A$ is an admissible $A$-valued uniform algebra on $K$ with $\fA=\p_N(K)$.
\end{theorem}

\begin{proof}
  Since $\pp(E,A)$ contains $0$-homogenous polynomials, we see that $\A$
  contains the constant functions, and since $\pp(E,A)$ contains the $1$-homogenous
  polynomials of the form $P=\psi \U$ ($\psi\in E^*$) we see that $\A$ separates points of $K$.
  To prove that $\A$ is an algebra, we show that it is closed under multiplication; that is,
  \begin{equation}\label{eqn:f,g-in-A-then-fg-in-A}
    f,g\in \A \Rightarrow fg\in \A.
  \end{equation}

  Given $f,g\in \A$, there exist sequences $(P_k)$ and $(Q_k)$ of polynomials in $\pp(E, A)$
  such that $P_k\to f$ and $Q_k\to g$, uniformly on $K$, from which we get $P_k Q_k \to fg$,
  uniformly on $K$. Therefore, \eqref{eqn:f,g-in-A-then-fg-in-A} reduces
  to the following;
  \begin{equation}\label{eqn:P,Q-in-P(E,A)-then-PQ-in-A}
    P, Q\in \pp(E, A) \Rightarrow PQ \in \A.
  \end{equation}

  Since every polynomial in $\pp(E,A)$ is a linear combination of homogenous polynomials,
  we may assume that $P\in \pp(^m E, A)$ and $Q\in \pp(^n E, A)$.

  Consider two cases;
  \begin{enumerate}
    \item $m,n\geq 1$,
    \item $m\geq1$ and $n=0$.
  \end{enumerate}

  In case (1), write $P = \sum_{i= 1}^\infty S_i^m$ and $Q = \sum_{j=1}^\infty T_j^n$,
  where $(S_i)$ and $(T_j)$ are sequences of operators in $\cL(E,A)$.
  By Proposition \ref{prop:norm-u-leq-M-normm}, these series converge uniformly on $K$;
  \begin{align*}
    \lim_{s\to\infty} \Bignorm{P- \sum_{i= 1}^s S_i^m}_K =
    \lim_{s\to\infty} \Bignorm{Q- \sum_{j= 1}^s T_j^n}_K  = 0.
  \end{align*}

  Therefore,
  \begin{equation}\label{eqn:PQ-is-approximated}
    \lim_{s\to\infty}\Bignorm{PQ- \sum_{i=1}^s S_i^m \sum_{j=1}^s T_j^n}_K
    = \lim_{s\to\infty}\Bignorm{PQ- \sum_{i=1}^s \sum_{j=1}^s S_i^m T_j^n}_K = 0.
  \end{equation}

  Applying polarization formula \eqref{eqn:polarization-ab}, we have
  \begin{align*}
    S_i^m T_j^n = \frac1{2^{m+n}(m+n)!}\sum_{\e_\ell=\pm 1}
     \e_1\dotsb\e_{m+n}\biggprn{\sum_{\ell=1}^m\e_\ell S_i + \sum_{\ell=m+1}^{m+n}\e_\ell T_j}^{m+n}.
  \end{align*}

  With suitable choice of scalars $\alpha_k$ and operators $U_k$, $1 \leq k \leq 2^{m+n}$,
  we have
  \begin{equation*}
    S_i^m T_j^n = \sum_{k=1}^{2^{m+n}} \alpha_k U_k^{m+n},
  \end{equation*}
  meaning that $S_i^mT_j^n\in \pp({}^{m+n}E,A)$. Now, \eqref{eqn:PQ-is-approximated}
  shows that $PQ$ is approximated (uniformly on $K$) by elements of $\pp({}^{m+n}E,A)$ and thus
  $PQ\in \A$.

  In case (2), write $P = \sum_{i= 1}^\infty S_i^m$ and $Q=b$, where $b\in A$ is a constant.
  Since $A$ is assumed to have an identity $\U$,
  the proof of Proposition 2.1 in \cite{Lourenco-a-class} shows that if
  $\lambda_k = \frac1{m^2} e^{\frac{2\pi ki}m}$ and $a_k = b + e^{\frac{2\pi ki}m}\U$,
  $1\leq k \leq m$, then $b = \lambda_1 a_1^m + \dotsb + \lambda_m a_m^m$.
  Therefore, if
  \begin{equation*}
    b_k = \frac{e^{\frac{2\pi k i}{m^2}}}{\sqrt[m]{m^2}}\Bigprn{b + e^{\frac{2\pi k i}{m}}\U},\
    k=1,2,\dotsc,m,
  \end{equation*}
  then $b = b_1^m + \dotsb + b_m^m$. Now, for every $x\in E$, we have
  \begin{align*}
    (PQ)(x)
        = P(x)b
      & = \sum_{j=1}^\infty \psi_j(x)^m \sum_{k=1}^m b_k^m \\
      & = \sum_{j=1}^\infty\sum_{k=1}^m (\psi_j(x)b_k)^m
        = \sum_{k=1}^m\sum_{j=1}^\infty(\psi_j(x) b_k)^m.
  \end{align*}

  Set $T_{jk} =\psi_j b_k$ and $P_k = \sum_{j=1}^\infty T_{jk}^m$.
  Then $P_k\in \pp({}^mE,A)$ and $PQ = P_1 +\dotsb+P_m$. Therefore, $PQ$
  is an $m$-homogenous polynomial in $\pp(E,A)$. In particular, $PQ\in \A$.

  Finally, we show that $\pp(K,A)$ is admissible; that is,
  \begin{equation}\label{eqn:phi-o-f-in-A}
    f\in\A, \phi\in\mis(A)\Rightarrow \phi\circ f \in \A.
  \end{equation}

  Given $f\in\A$, let $(P_k)$ be a sequence of polynomials in $\pp(E,A)$ such that
  $P_k\to f$ uniformly on $K$. For every $\phi\in \mis(A)$, since $\|\phi\|\leq 1$, we have
  \begin{equation*}
    \|\phi\circ P_k - \phi \circ f\|_K \leq \|P_k-f\|_K,
  \end{equation*}
  and thus $\phi\circ P_k \to \phi\circ f$ uniformly on $K$.
  Therefore, \eqref{eqn:phi-o-f-in-A} reduces to the following;
  \begin{equation}\label{eqn:phi-o-P-in-A}
    P\in \pp(E,A), \phi\in\mis(A)\Rightarrow \phi\circ P \in \A.
  \end{equation}

  Given $P\in \pp(E,A)$, since it is a linear combination of homogenous polynomials,
  we may assume that $P$ itself is an $n$-homogenous polynomial and write
  $P=\sum_{i=1}^\infty T_i^n$, for $T_i\in \cL(E,A)$. Then, for every $x\in E$,
   \begin{align*}
    (\phi \circ P)(x)
      & = \phi\Bigprn{\sum_{i=1}^\infty T_i^n(x)} = \sum_{i=1}^\infty \phi\bigprn{T_i(x)^n} \\
      & = \sum_{i=1}^\infty \phi\bigprn{T_i(x)}^n
        = \sum_{i=1}^\infty\bigprn{\phi \circ T_i(x)}^n.
   \end{align*}

   We see that $\phi\circ P$ is a nuclear polynomial (i.e., $\phi\circ P \in \p_N(E)$).
   Set $\psi_i=\phi\circ T_i$ and $S_i=\psi_i\U$. Then $S_i\in\cL(E,A)$, $\|S_i\| \leq \|T_i\|$,
   and
   \begin{equation*}
     (\phi\circ P)\U = \sum_{i=1}^\infty (\psi_i\U)^n = \sum_{i=1}^\infty S_i^n \in \pp(K,A).
   \end{equation*}
   We conclude that $\pp(K, A)$ is admissible with $\fA=\p_N(K)$.
\end{proof}

\subsection{Representing as tensor product}

We now investigate conditions that imply the equality $\pp(K,A)=\p_N(K)\htp_\e A$.
In the following, $I:E\to E$ is the identity operator.

\begin{theorem}\label{thm:pp(K,A)=pp(K)-tp-A-iff}
  Let $K$ be a compact set in the Banach space $E$.
  The following statements are equivalent.
  \begin{enumerate}[\upshape(i)]
    \item\label{item:for-every-A}
      $\pp(K,A)= \p_N(K)\htp_\e A$, for every unital Banach algebra $A$.

    \item\label{item:just-for-E}
      $\pp(K,E)= \p_N(K)\htp_\e E$.

    \item\label{item:I-in-P-tp-E}
      $I\in\p_N(K)\htp_\e E$.
  \end{enumerate}
\end{theorem}

\begin{proof}
 The implication \eqref{item:for-every-A} $\Rightarrow$ \eqref{item:just-for-E} is trivial
 (see Remark \ref{rem:E-is-an-algebra}).

 The implication \eqref{item:just-for-E} $\Rightarrow$ \eqref{item:I-in-P-tp-E} is also trivial
 since always $I\in\pp(K,E)$.

 We prove the implication \eqref{item:I-in-P-tp-E} $\Rightarrow$ \eqref{item:for-every-A}.
 Let $A$ be a unital Banach algebra. In view of Remark \ref{rem:injective-tesor-prod}
 and Theorem \ref{thm:P(K,A)-is-an-algebra},
 we always have $\p_N(K)\htp_\e A \subset \pp(K,A)$. To prove the reverse inclusion,
 since $\overline{\pp_0(K,A)}=\pp(K,A)$, we just need  to show that
 $\pp_0(K,A) \subset \p_N(K) \htp_\e A$; this is, $P \in \p_N(K) \htp_\e A$,
 for every polynomial $P\in\pp(E,A)$.

 We argue by induction on $n=\deg(P)$.

 The base case, $n=0$, trivially hold. In fact, if $\deg(P)=0$ then $P$ is a constant polynomial
 and belongs to $\p_N(K)\htp_\e A$.

 For the inductive step, take $n\in\N$ and assume that
 \begin{equation}\label{inductive-step}
   \text{every polynomial $Q$ with $\deg(Q)<n$, belongs to $\p_N(K)\htp_\e A$},
 \end{equation}
 and let $P$ be a polynomial with $\deg(P)=n$. Subtracting from $P$ a polynomial $Q$ with
 $\deg(Q)\leq n-1$, if necessary, we may assume that $P$ is an $n$-homogenous polynomial
 and that $P=\hat T$, for some $T\in \cL_s({}^nE,A)$. Let $\e>0$. By the assumption
 $I\in\p_N(K)\htp_\e E$. Therefore, by Remark \ref{rem:injective-tesor-prod},
 there exist functions $f_1,\dotsc,f_m$ in $\p_N(K)$ and vectors $a_1,\dotsc,a_m$ in $E$ such that
 \begin{equation}\label{eqn:I-app-by-f-tp-a}
   \Bignorm{I(x)-\sum_{i=1}^m f_i(x)a_i} \leq \e,\ x\in K.
 \end{equation}

 Using notation \eqref{eqn:notation-T(x,..,x)}
 and Leibniz formula \eqref{eqn:Leibniz-formula}, we get
 \begin{equation}\label{eqn:T(x-sum Pa)}
   T\Bigprn{\bigprn{x-\sum_{i=1}^m f_i(x)a_i}^n}
       = \sum_{k=0}^n \binom{n}{k} T\Bigprn{x^k, \bigprn{-\sum_{i=1}^mf_i(x)a_i}^{n-k}}.
 \end{equation}

 Defined, for $0 \leq k \leq n-1$ and $x\in K$,
   \begin{equation*}
      g_k(x) = \binom nkT\Bigprn{x^k, \bigprn{-\sum_{i=1}^m f_i(x)a_i}^{n-k}}.
   \end{equation*}

  Indeed,
 \begin{align*}
   g_{n-1}(x) & = -n \sum_{i=1}^m f_i(x) T(x^{n-1},a_i), \\
   g_{n-2}(x) & = n(n-1)\sum_{i,j=1}^m f_i(x) f_j(x) T(x^{n-2},a_i,a_j), \\
              & \ \ \vdots \\
   g_0(x) & = (-1)^n \sum_\alpha \frac {n!}{\alpha!}
                 \prod_{i=1}^m f_i(x)^{\alpha_i} 
                   T(a_1^{\alpha_1},a_2^{\alpha_2},\dotsc,a_m^{\alpha_m}),
 \end{align*}
 where the last summation is taken over all multi-indices $\alpha=(\alpha_1,\dotsc,\alpha_m)$
 in $\N_0^m$ such that $|\alpha|=\alpha_1+\dotsb+\alpha_m=n$.

 We see that each $g_j$, $0 \leq j \leq n-1$, is an algebraic combination
 (multiplication and addition) of functions $f_1,f_2, \dotsc,f_m$ in $\p_N(K)$,
 and some polynomials of degree $j$ which, by assumption \eqref{inductive-step},
 belong to $\p_N(K)\htp_\e A$. Therefore, $g_0,g_1,\dotsc,g_{n-1} \in \p_N(K) \hat \tp_\e A$.

 On the other hand, for every $x\in K$,
 \begin{align*}
   \Bignorm{P(x) + \sum_{j=0}^{n-1} g_j(x)}
   & =  \Bignorm{T(x^n) + \sum_{k=0}^{n-1} \binom{n}{k} T\Bigprn{x^k, \bigprn{-\sum_{i=1}^mf_i(x)a_i}^{n-k}}} \\
    & =  \Bignorm{T\Bigprn{\bigprn{x - \sum_{i=1}^m f_i(x)a_i}^n}} \\
    &  \leq \|T\| \Bignorm{x - \sum_{i=1}^m f_i(x)a_i}^n \leq \|T\|\e^n,
 \end{align*}
 where $\|T\|$ is the operator norm of $T$ in $\cL_s(E,A)$.
 Since $\e$ is arbitrary, this means that $P$ is approximated uniformly on $K$
 by functions in $\p_N(K)\htp_\e A$. Therefore, $P\in\p_N(K)\htp_\e A$ and
 the inductive argument is complete.

 We conclude that $\pp(K,A) = \p_N(K)\htp_\e A$.
\end{proof}

Recall (e.g., \cite[Definition 27.3]{Mujica}) that a Banach space $E$ has the \emph{approximation property}
if, for every $\e>0$ and every compact set $K$ in $E$, there exists a finite rank operator $T:E\to E$
such that $\|T(x)-x\|<\e$, for every $x\in K$.

\begin{proposition}\label{prop:if-E-has-AP}
  Suppose that the Banach space $E$ has the approximation property.
  Then $I\in \p_N(K)\hat \tp_\e E$, for any compact set $K$ in $E$.
\end{proposition}

\begin{proof}
  Let $\e>0$. By the approximation property, there is a finite-rank operator
  $T:E\to E$ such that $\|I-T\|_K \leq \e$. The finite-rank operator $T$ can be represented
  in a form
    \begin{equation*}
    T=\psi_1 a_1 + \dotsb + \psi_m a_m,
    \end{equation*}
  where $\psi_i \in E^*$, $a_i\in E$,
  $1 \leq i \leq m$. This means that $T\in \p_N(K) \htp_\e E$. Since $\e$ is arbitrary,
  we conclude that $I\in \p_N(K)\hat \tp_\e E$.
\end{proof}

\newcommand{\cc}{\mathrm{cc}}

For the Banach algebra $A$, let us denote by $A^*_\cc$ the space $A^*$ equipped with the topology of
compact convergence (i.e., the topology of uniform convergence on compact subsets of $A$). Given a set $S$
in $A^*$, we say that \emph{$S$ generates $A^*_\cc$} if $\ip{S}$ is dense in $A^*_\cc$, where $\ip{S}$ denotes
the linear span of $S$ in $A^*$.

\begin{proposition}\label{prop:if-A-has-AP}
  Suppose that the Banach algebra $A$ has the approximation property. If $\mis(A)$ generates $A^*_\cc$,
  then $\pp(K,A)=\p_N(K)\hat \tp_\e A$, for any compact set $K$ in any Banach space $E$.
\end{proposition}

\begin{proof}
  Let $K$ be a compact set in a Banach space $E$ and take a function $f\in \pp(K,A)$.
  First, we show that $\phi\circ f\in \p_N(K)$, for every $\phi\in A^*$.
  By Theorem \ref{thm:P(K,A)-is-an-algebra}, we have $\phi\circ f \in \p_N(K)$, for
  every $\phi\in \mis(A)$, whence $\phi\circ f \in \p_N(K)$, for every $\phi\in \ip{\mis(A)}$,
  the linear span of $\mis(A)$ in $A^*$. Since $\mis(A)$ generates $A^*_\cc$, given $\phi\in A^*$, there exists
  a net $(\phi_\alpha)$ in $\ip{\mis(A)}$ such that $\phi_\alpha \to \phi$ in the compact
  convergence topology of $A^*$. The set $f(K)$ is compact in $A$, and thus $\phi_\alpha \to \phi$
  uniformly on $f(K)$. This implies that $\phi_\alpha \circ f \to \phi\circ f$ uniformly on $K$.
  Since $\phi_\alpha \circ f \in\p_N(K)$, for all $\alpha$, we get $\phi\circ f\in \p_N(K)$.

  Now, let $\e>0$. Since $A$ has the approximation property and $f(K)$ is compact in $A$,
  there exists a finite rank operator $T=\sum_{i=1}^m \phi_i b_i$, with $\phi_i \in A^*$,
  $b_i \in A$, $1\leq i \leq m$, $m\in \N$, such that
  \begin{equation*}
    \|y-T(y)\| = \Bignorm{y - \sum_{i=1}^m \phi_i(y)b_i} \leq \e,\ y\in f(K).
  \end{equation*}
  Replacing $y$ with $f(x)$, $x\in K$, we get
  \begin{equation*}
    \Bignorm{f(x) - \sum_{i=1}^m \phi_i\circ f(x)b_i} \leq \e,\ x\in K.
  \end{equation*}

  From the first part of the proof, we get $\sum_{i=1}^m(\phi_i\circ f)b_i\in \p_N(K)\htp_\e A$.
  Since $\e>0$ is arbitrary, we have $f\in \p_N(K) \htp_\e A$.
\end{proof}

To support the above result, we present some examples.

\begin{example}
  Suppose that $X$ is a compact Hausdorff space and that $A=\c(X)$. It is well-known that
  $A$ has the approximation property and that $\mis(A)=\set{\delta_x:x\in X}$, where
  $\delta_x:f\mapsto f(x)$ is the point mass measure at $x$. Indeed, $\mis(A)$ coincides
  with the set of extreme points of the unit ball $A_1^*$. By the Krein-Millman theorem,
  $A_1^*$ equals the weak* closed convex hull of $\mis(A)$. Therefore, given
  $\phi\in A_1^*$, there exists a net $(\phi_\alpha)$ in the convex hull of $\mis(A)$
  such that $\phi_\alpha\to \phi$, in the weak* topology. Since $(\phi_\alpha)$ is bounded,
  we get $\phi_\alpha\to \phi$, uniformly on compact sets. This means that $\phi_\alpha\to \phi$ in $A^*_\cc$,
  and thus $\mis(A)$ generates $A^*_\cc$. Now, by Proposition \ref{prop:if-A-has-AP}, we have
  $\pp(K,\c(X))=\p_N(K)\htp_\e \c(X)$, for any compact set $K$ in a Banach space $E$.
  It is worth noting that
  \begin{equation*}
    \p_N(K)\htp_\e \c(X) = \c(X)\htp_\e \p_N(K) = \c(X,\p_N(K)).
  \end{equation*}
\end{example}

\begin{example} \label{exa:lp(S)}
  Let $S$ be any nonempty set, and $A=\ell^p(S)$, for some $p\in (1,\infty)$.
  Then $A^*=\ell^q(S)$ with $1/p+1/q=1$. Define multiplication on $A$ point-wise; that is, $(fg)(s) = f(s) g(s)$, for all
  $s\in S$. It is a matter of calculation to verify that
  \[
    \sum_{s\in S} |f(s)|^p|g(s)|^p \leq \sum_{s\in S} |f(s)|^p \sum_{s\in S} |g(s)|^p
    \quad (f,g\in A).
  \]

  Therefore $(A,\enorm_p)$ is a commutative
  Banach algebra. For every $s\in S$, the evaluation homomorphism $\phi_s:f\mapsto f(s)$ is a character of $A$.
  Conversely, let $\phi:A\to \C$ be a character. Suppose that $\set{\chi_s:s\in S}$ is the collection of all
  the characteristic functions on $S$. Then, given $f\in A$, we have $f = \sum_{s\in S} f(s) \chi_s$,
  and thus $\phi(f) = \sum_{s\in S}f(s)\phi(\chi_s)$. Since $\phi\neq0$,
  there must be a point $s_0\in S$ such that $\phi(\chi_{s_0})\neq 0$. If $s\neq s_0$,
  then $\phi(\chi_s)\phi(\chi_{s_0}) = \phi(\chi_s\chi_{s_0})=0$ so that
  $\phi(\chi_s)=0$. Also $\phi(\chi_{s_0})^2 = \phi(\chi_{s_0})$ and thus $\phi(\chi_{s_0})=1$.
  We therefore have
  \[
    \phi(f) = \sum_{s\in S}f(s)\phi(\chi_s) = f(s_0) = \phi_{s_0}(f).
  \]

  We conclude that $\mis(A)=\set{\phi_s:s\in S}$. The fact that $\ip{\mis(A)}$
  is dense in $A^*=\ell^q(S)$ in the norm topology, yields that $\mis(A)$
  generates $A^*_\cc$. Moreover, if $S$ is countable then $\ell^p(S)$
  has a Schauder basis and thus it has the approximation property.
\end{example}

\subsection{The character space}

Let $K$ be a compact set in the Banach space $E$.
The character space of a function algebra on $K$ generated by a certain class $\p$
of polynomials is closely related to the polynomially convex hull of $K$ with
respect to the given class $\p$.

\begin{definition}
   The \emph{nuclear polynomially convex hull} of $K$ is defined as
   \begin{equation}\label{eqn:nuclear-poly-convex-hull}
       \hat K_N = \set{a\in E: |P(a)| \leq \|P\|_K, P\in \p_N(E)}.
   \end{equation}
   It is said that $K$ is \emph{nuclear polynomially convex} if $K=\hat K_N$.
\end{definition}

\begin{theorem}\label{thm:mis(P-N(K))}
  The character space of $\p_N(K)$ is homeomorphic to $\hat K_N$, and $\p_N(K)$
  is isometrically isomorphic to $\p_N(\hat K_N)$.
\end{theorem}

\begin{proof}
  Let $a\in \hat K_N$. Given $f\in\p_N(K)$, there is a sequence $(P_k)$ of polynomials
  in $\p_{N_0}(K)$ such that $P_k\to f$, uniformly on $K$, and thus
   \begin{equation*}
    |P_k(a)-P_j(a)| \leq \|P_k-P_j\|_K \to0 \quad (k,j\to\infty).
   \end{equation*}

  This means that $(P_k(a))$ is a Cauchy sequence in $\C$.
  Define $\phi_a(f)=\lim\limits_{k\to\infty}P_k(a)$. If $(Q_k)$ is another sequence of polynomials
  in $\p_{N_0}(K)$ such that $Q_k\to f$, uniformly on $K$, then $|Q_k(a)-P_k(a)| \leq \|Q_k-P_k\|_K\to0$.
  This shows that $\phi_a(f)$ is well-defined. An standard argument shows that
  $\phi_a(f+g)=\phi_a(f)+\phi_a(g)$ and $\phi_a(fg)=\phi_a(f)\phi_a(g)$, for all $f,g\in\p_N(K)$.
  Therefore, $\phi_a\in\mis(\p_N(K))$.

  Conversely, assume that $\phi:\p_N(K)\to \C$ is a character.
  Consider the dual space $E^*$ as a subspace of $\p_N(E)$ consisting of $1$-homogenous polynomials.
  Then, the restriction of $\phi$ to $E^*$ is a linear functional on $E^*$. We show that $\phi$ is
  weak* continuous on norm bounded subsets of $E^*$. Let $(\psi_\alpha)$ be a bounded net in $E^*$
  that converges, in the weak* topology, to some $\psi_0\in E^*$.
  Then $\psi_\alpha\to \psi_0$ uniformly on $K$. Since $\phi$ is continuous
  with respect to $\enorm_K$, we get $\phi(\psi_\alpha)\to \phi(\psi_0)$. This shows that $\phi$
  is weak* continuous on bounded subsets of $E^*$, as desired. By \cite[Corollary 4, page 250]{Horvath},
  $\phi$ is weak* continuous on $E^*$ and thus there is $a\in E$ such that
  $\phi(\psi) = \psi(a)$ for all $\psi\in E^*$. Now, take an $n$-homogenous nuclear polynomial
  $P=\sum_{i=1}^\infty \psi_i^n$, with $\psi_i\in E^*$ and $\sum_{i=1}^\infty\|\psi_i\|^n<\infty$.
  By Proposition \ref{prop:norm-u-leq-M-normm}, the series converges uniformly on $K$, and thus
  \begin{equation*}
    \phi(P) = \phi\Bigprn{\lim_{s\to\infty}\sum_{i=1}^s \psi_i^n}
     = \lim_{s\to\infty}\sum_{i=1}^s \phi(\psi_i)^n
     = \lim_{s\to\infty}\sum_{i=1}^s \psi_i(a)^n = P(a).
  \end{equation*}

  Note that $|P(a)|= |\phi(P)|\leq \|P\|_K$, for every $P\in \p_N(E)$ which shows that $a\in \hat K_N$.
  Thus $\phi = \phi_a$ on $\p_{N_0}(K)$, a dense subspace of $\p_N(K)$, whence $\phi=\phi_a$
  on $\p_N(K)$. Finally, the mapping $\hat K_N\to \mis(\p_N(K))$, $a\mapsto \phi_a$, is
  an embedding of $K$ onto $\mis(\p_N(K))$; see \cite[Chapter 4]{Dales}.
\end{proof}

We conclude this paper with the following result on the character space of $\pp(K,A)$.

\begin{theorem}
  The character space $\mis(\pp(K,A))$ contains $\hat K_N \times\mis(A)$ as a
  closed subset. If either of the conditions in Theorem $\ref{thm:pp(K,A)=pp(K)-tp-A-iff}$,
  Proposition $\ref{prop:if-E-has-AP}$ or Proposition $\ref{prop:if-A-has-AP}$ hold,
  then $\mis(\pp(K,A))=\hat K_N \times\mis(A)$.
\end{theorem}

\begin{proof}
  It follows from previous results and a theorem of Tomiyama
  in \cite{Tomiyama}.
\end{proof}



\begin{thebibliography}{99}

\bibitem{Abtahi-BJMA}
 Abtahi M.
 \textit{Vector-valued characters on vector-valued function algebras},
 Banach J. Math. Anal. 2016, \textbf{10}(3), 608--620.

\bibitem{Carando-extendible}
 Carando D.
 \textit{Extendible polynomials on Banach spaces},
 J. Math. Anal. Appl. 1999, \textbf{233}(1), 359--372.

\bibitem{Carando-duality}
 Carando D., Dimant V.
 \textit{Duality in spaces of nuclear and integral polynomials},
 J. Math. Anal. Appl. 2000, \textbf{241}(1), 107--121.

\bibitem{Cilia-nuclear}
 Cilia R., Guti\'errez J.M.
 \textit{Nuclear and integral polynomials},
 J. Aust. Math. Soc. 2004, \textbf{76}(2), 269--280.

\bibitem{Dales}
 Dales H.G.
 Banach Algebras and Automatic Continuity.
 London Mathematical Society Monographs, New Series 24, Oxford
 Science Publications, Clarendon Press, Oxford University Press,
 New York, 2000.


\bibitem{Dineen}
 Dineen S.
 Complex Analysis on Infinite Dimensional Spaces.
 Springer-Verlag, London, 1999.

\bibitem{Galindo-PGLO}
 Galindo P., Lourenco M., Moraes L.
 \textit{Polynomials generated by linear operators},
 Proc. Amer. Math. Soc. 2004, \textbf{132}(10), 2917--2927.

\bibitem{Hausner}
  Hausner A.
  \textit{Ideals in a certain Banach algebra},
  Proc. Amer. Math. Soc. 1957, \textbf{8}(2), 246--249.

\bibitem{Horvath}
  Horv\'ath J.
  Topological vector spaces and distributions. Vol. I,
  Addison-Wesley, Reading, Mass., 1966.

\bibitem{Lourenco-a-class}
 Lourenco M.L., Moraes L.A.
 \textit{A class of polynomials from Banach spaces into Banach algebras}.
 Publ. Res. Inst. Math. Sci., Kyoto Univ. 2001,
 \textbf{37}(4), 521--529.

\bibitem{Mujica}
 Mujica J.
 Complex Analysis in Banach Spaces. North-Holland Mathematical Studies 120,
 North-Holland, Amesterdam, 1986.

\bibitem{Nikou-Ofarrell}
 Nikou A., O\'farrell A.G.
 \textit{Banach algebras of vector-valued functions},
 Glasgow Math. J. 2014, \textbf{56}, 419--246.

\bibitem{Prolla}
 Prolla J.B.
 Approximation of Vector-Valued Functions,
 North-Holland, Amsterdam, NewYork, Oxford, 1977.

\bibitem{Tomiyama}
 Tomiyama J.
 \textit{Tensor products of commutative Banach algebras},
 Tohoku Math. J. 1960, \textbf{12}, 147--154.

\bibitem{Zal-survay}
 Zalduendo I.
 \textit{Extending Polynomials On Banach Spaces A Survey},
 Union Mathematic Argentina, 2005.

\end{thebibliography}
\end{document}